\documentclass[10pt]{article}
\usepackage{times,amsfonts,amsmath,amssymb,exscale,mathrsfs}
\usepackage{pst-all}
\usepackage{graphicx}
\usepackage{xcolor} 
\usepackage{epsfig}
\newtheorem{theo}{Theorem}
\newtheorem{prop}[theo]{Proposition}
\newtheorem{lemma}[theo]{Lemma}
\newenvironment{proof}{\noindent{\bf Proof:}}{$\Box$\medskip}
\newcommand{\A}{\mathcal{A}}
\newcommand{\E}{\mathbb{E}}
\newcommand{\G}{\mathcal{G}_{sh}}
\newcommand{\N}{\mathbb{N}}
\newcommand{\Nh}{(1/2)\mathbb{N}}
\newcommand{\R}{\mathbb{R}}
\newcommand{\RW}{\mathbb{R}^{\mathcal{W}}}
\newcommand{\W}{\mathcal{W}}

\newcommand{\one}{{\, 1\!\! 1\,}}
\newcommand{\sh}{{\,\scriptstyle \sqcup\!\sqcup\,}}
\title{A technique for studying  strong and weak local errors of splitting stochastic integrators}
\author{A. Alamo\footnote{Departamento de Matem\'atica Aplicada e IMUVA, Facultad de Ciencias, Universidad de Valladolid,  Spain. Email: alamozapaterouva@gmail.com}
\  and J.M. Sanz-Serna\footnote{Departamento de Matem\'aticas, Universidad Carlos III de Madrid, Avenida de la Universidad 30, E-28911 Legan\'es (Madrid), Spain. Email: jmsanzserna@gmail.com}}
\date{\today}
\begin{document}
\maketitle

\begin{abstract}
We present a technique, based on so-called  word series, to write down in a systematic way  expansions of the strong and weak local errors of splitting algorithms for the integration of Stratonovich stochastic differential equations. Those expansions immediately lead to the corresponding order conditions. Word series are similar to, but simpler than, the B-series used to analyze Runge-Kutta and other one-step integrators. The suggested approach makes it unnecessary to use  the Baker-Campbell-Hausdorff formula.
As an application, we compare two splitting algorithms recently considered by Leimkuhler and Matthews to integrate the Langevin equations. The word series method bears out clearly reasons for the advantages of one algorithm over the other.
\end{abstract}
\bigskip

\bigskip


\section{Introduction}

We present a technique, based on so-called {\em word series}, to write down in a systematic way  expansions of the strong and weak local errors of splitting algorithms for the integration of Stratonovich stochastic differential equations (SDEs). Those expansions immediately lead to the corresponding order conditions without any need to use the Baker-Campbell-Haussdord formula. As an application we compare two splitting algorithms  recently considered by Leimkuhler and Matthews \cite{Leim}, \cite{Leim2}, \cite{Leim3} to integrate the Langevin equations.

The approach taken in this article may be seen as patterned after the seminal work of Butcher \cite{butcher63} on the {\em combinatorics} of the order conditions for Runge-Kutta deterministic integrators. As is well known, in the theory developed by Butcher, the numerical and true solution are expanded with the help of vector-valued mappings called elementary differentials. In the expansions, the elementary differentials are weighted by so-called  elementary weights. These are real numbers that change with the integrator but are independent of the system being integrated.
There are an elementary differential and an elementary weight for each rooted tree and both are easily written down because their structure is a transcription of that of the rooted tree. The elementary differentials change with the differential system being integrated but are common for all Runge-Kutta integrators and also for the true solution; this has important implications because when designing new integrators or  comparing different integrators one may focus on the elementary weights. B-series \cite{HW}, series of elementary differential with arbitrary coefficients, are a way of systematizing  Butcher's approach and extending it to more general integrators. A key result  in \cite{HW} is the rule for composing B-series. B-series have found many applications in numerical analysis, in particular in relation with modified equations \cite{aust} and geometric integration \cite{canonical}, \cite{ssc}, \cite{hlw}. For applications of B-series outside numerical mathematics see \cite{part1}, \cite{part2}.
Burrage and Burrage \cite{burrage} have  analyzed \`{a} la Butcher weak and strong errors of Runge-Kutta integrators for SDEs.
The paper \cite{china} surveys the history of these developments.

The importance of splitting algorithms \cite{survey} has been increasing in recent years, essentially as a consequence of their capability of exploiting the structure of the problem being integrated. In the deterministic case, there are several ways of investigating  the  consistency properties of a splitting integrator:
\begin{itemize}
\item The best known technique, described  in e.g.\ \cite{gi},  applies  the { Baker-Campbell-Hausdorff} formula. This method has several shortcomings, including the {\em huge combinatorial complexity} of the BCH formula itself (see \cite{maka} and \cite{phil} for a discussion).
\item An approach that parallels  Butcher's treatment of Runge-Kutta formulas has been introduced in \cite{phil} (a summary is available in \cite[Section III.3]{hlw}). As in Butcher's work, the approach is based on the use of rooted trees. The {\em B-series} expansions found in this way are also made of elementary differentials and scalar coefficients.
\item More recently {\em word series} expansions \cite{anderfocm}, \cite{orlando}, \cite{part3}, \cite{words}, \cite{juanluis}, \cite{k} have been suggested as an alternative to B-series. The scope of applicability of word series is narrower than that of B-series: splitting methods may be treated with word series but Runge-Kutta formulas may not. When applicable, word series are more convenient than B-series. They are more compact than B-series and have a composition rule (see Theorem~\ref{th:rules}) much simpler than the recipe used to compose B-series.
\end{itemize}

In the present work we extend the third technique above to cater for splitting integrators for Stratonovich SDEs, thus avoiding the complicated combinatorics involved in the BCH formula. In Section 2 we present the tools required in the rest of the article. In Section 3 we show how to expand a composition of exact solutions by using the formula for composing word series. For clarity, the idea is presented in the deterministic case where several complications of the Stratonovich scenario are absent.
In Section 4, we provide formulas for the expansion of both strong and weak local errors and write down the associated order conditions. In Section 5 the material is applied to the  case of Langevin dynamics. Leimkuhler and Matthews \cite{Leim}, \cite{Leim2} have considered two closely related splitting algorithms and found, numerically and theoretically, that  one is clearly superior to the other. We show that a word series analysis identifies {\em additional} reasons for that superiority.  Section 6 describes additional possible uses of word series in the analysis of SDE integrators.

It is well known that error expansions like those considered in Section 4 in general do not converge. This does not diminish their usefulness: by truncating the series one obtains the Taylor polynomials that are needed to write down the order conditions. Of course when {\em bounds} of the weak or strong local error are required it is necessary to estimate the remainder term in the error expansion. Although the emphasis of this article is in the combinatorics of the expansion rather than on error estimates, we have included an Appendix that illustrates how to derive error bounds for word series expansions (cf.\ \cite{orlando}, \cite{part3}, \cite{words}).

For simplicity, except in the Appendix, all mappings are assumed to be indefinitely differentiable. Of course, when that is not the case, the formulas presented below only make sense up to the order where the  derivatives that appear exist.

\section{Preliminaries}

In this section we describe word series. The presentation is very concise. References are grouped in Section~\ref{sec:discussion}.

\subsection{Words}
Let \( \A \)  be a finite set, which we shall call  the {\em alphabet}.  The elements \( a\in\A \) are called  {\em letters}. A {\em word} \( w\) is an arbitrary finite sequence of letters \( a_1a_2\dots a_n\), \(a_i\in\A\). We  denote by $\mathcal{W}$ the set of all words, including the empty word $\emptyset$, i.e.\ the word with zero letters. No distinction is made between the letter \( a \) and the word having \( a \) as its only letter, so that \( \A \) is seen a subset of \( \W \).

We work with mappings  \(\delta:\W\to\R\) and use the notation \(\delta_w\) to refer to the real value that \( \delta \) takes at \( w\in\W \). The set  \( \RW \) consists of all such mappings.
Given \( \delta, \delta^\prime \in  \RW\), we associate with them  their {\em convolution product}  \(\delta \star \delta^\prime \in \RW\), defined by \((\delta\star\delta^\prime )_{\emptyset}=\delta_\emptyset \delta^\prime_\emptyset\) and, for nonempty words,
\[
 (\delta \star \delta^\prime)_{a_1a_2\dots a_n}
 =\delta_{\emptyset}\delta^\prime_{a_1a_2\dots a_n}+\sum_{j=1}^{n-1}\delta_{a_1a_2\dots a_j}\delta^\prime_{a_{j+1}\dots a_n}+\delta_{a_1a_2\dots a_n}\delta^\prime_{\emptyset}.
\]
Note that in the right-hand side there is a term for each of the ways in which \( {a_1a_2\dots a_n}\) may be split into two subwords (in more technical language {\em deconcatenated} into two subwords). The operation \( \star \)  is not commutative, but it is
is associative; to find the value of \( \delta \star \delta^\prime\star  \delta^{\prime\prime} = (\delta\star\delta ^\prime) \star\delta^{\prime\prime} =\delta\star  (\delta^\prime \star\delta^{\prime\prime}) \) at a word \( w\) we sum all the values \(\delta_v\delta_{v^\prime}^\prime\delta_{v^{\prime\prime}}^{\prime\prime}\) corresponding to triples \( v, v^\prime,v^{\prime\prime}\) that concatenated yield \( w \).
The element \( \one\in\RW\) specified by   \(\one_\emptyset = 1\) and \(\one_w = 0\) for each nonempty word $w$ is the unit of the operation \(\star\).

Given two words \( w \) and \( w^\prime \) with \( m \) and \( n \) letters respectively, their {\em shuffle  product} \(w\sh w^\prime\)  is the formal sum  of the \((m+n)!/(m!n!)\) words with \(m+n\) letters that may be obtained by interleaving the letters of \(w\) and \( w^\prime \) while preserving the order in which  the letters appear in \( w \) and \( w^\prime \). For instance,  \(a\sh b=ab+ba\), \( a\sh a = aa+ aa = 2aa\), \(ab\sh c = abc+acb+cab \), \( ab\sh cd=abcd+acbd+cabd+acdb+cadb+cdab\).


We shall denote by \( \G \) ({\em sh} for shuffle) the subset of \( \RW \) that comprises all the elements \( \gamma\in\RW \) satisfiying the so-called {\em shuffle relations:} \( \gamma_\emptyset=1 \) and, for each pair of words \( w,w^\prime \), if
\begin{equation}\label{eq:sh}
w \sh\ w^\prime=\sum_{j}w_j
\end{equation}
then
\begin{equation}\label{eq:sh2}
\gamma_{w}\gamma_{w^\prime}=\sum_j \gamma_{w_j}.
\end{equation}
For instance, \( \gamma_a\gamma_b = \gamma_{ab}+\gamma_{ba} \), \( \gamma^2_a = 2 \gamma_{aa} \),
\( \gamma_{ab}\gamma_ c = \gamma_{cab}+\gamma_{acb}+\gamma_{abc} \), etc.
For the convolution product, $\G$  is a {\em noncommutative   group} with unit \( \one \).

\subsection{Word series}\label{sec:words}
Assume now that for each letter \( a\in\A \), \( f_a:\R^d\to\R^d \) is a map. With every word \( w\in\W \), we  associate  a   {\em word basis function} \( f_w:\R^d\to\R^d\).  If \(w=a_1a_2\dots a_n\), \(n>1\), then \(f_w\) is defined recursively by
\begin{equation}\label{eq:wbs}
f_{a_1a_2\dots a_n}(x)=\big(\partial_x f_{a_2\dots a_n}(x)\big)f_{a_1}(x),
\end{equation}
where \( \partial_x f_{a_2\dots a_n}(x)\) denotes  the value at the point $x$ of the Jacobian matrix of $f_{a_2\dots a_n}$. For the empty word,   $f_\emptyset$ is simply the identity map \(x\mapsto x\).
With every $\delta\in\RW$ we associate a  {\em word series}. This is the formal series
\[
W_{\delta}(x)=\sum_{w\in\mathcal{W}}\delta_{w}f_{w}(x).
\]
The $\delta_w$ are the {\em coefficients} of the series. The notation \(W_{\delta}(x)\) does not incorporate the dependence on the \(f_a\), which are given once and for all.

As a very important example, consider the $d$-dimensional initial value problem
\begin{equation}\label{pvideter}
\frac{d}{dt}x=\sum_{a\in\A}{\lambda_a(t)f_a(x)}, \qquad x(t_0)=x_0
\end{equation}
where, for each $a\in\A$, $\lambda_a$ is a real-valued function of \(t\).
For each   \(t\), the solution value \(x(t)\in\R^d\) has a word series expansion
\begin{equation}\label{eq:solution}
x(t)=W_{\alpha(t;t_0)}(x_0)= \sum_{w\in\W}{\alpha_w(t;t_0)f_w\big(x_0\big)},
\end{equation}
with  coefficients  given by
\begin{equation}\label{eq:alpha1}
\alpha_\emptyset(t;t_0) = 1,\qquad \alpha_{a}(t;t_0)=\int_{t_0}^t\lambda_{a}(s)\,ds,\quad a \in \A,
\end{equation}
and, for words with \( n>1 \) letters, recursively,
\begin{equation}\label{eq:alpha2}
  \alpha_{a_1a_2\cdots a_n}(t;t_0)=\int_{t_0}^t\alpha_{a_1a_2\cdots a_{n-1}}(s;t_0)\lambda_{a_n}(s)\,ds.
\end{equation}
Thus, for a word with \(n>0\) letters, \(\alpha_w(t;t_0)\)  is an \(n\)-fold iterated integral or, equivalently, an integral over a simplex in \(\R^n\).

As we shall see later, for splitting numerical integrators, the numerical solution after a single step also possesses a word series expansion.

For future reference we point out that, as \( t \rightarrow t_0\), for each word of \( n\) letters,
\begin{equation}\label{eq:alphabound}
\alpha_{a_1a_2\dots a_n}(t;t_0) = \mathcal{O}\big((t-t_0)^n)
\end{equation}

In the simplest case where the alphabet consists of a single letter \( \A = \{a\} \) and
\( \lambda_a(t) = 1 \) for each \(t\),  there is one word \( a^n=a\dots a\) with \( n \) letters, \( n = 0,1, \dots\), and the corresponding coefficient is
\begin{equation}\label{eq:taylor}
\alpha_{a^n} = \frac{(t-t_0)^n}{n!};
 \end{equation}
 the word series representation \eqref{eq:solution} just coincides with the standard Taylor expansion of \( x(t) \) around \( t_0\)with the derivatives of \( x\)  expressed by means of the mapping \(f_a\), e.g.
\begin{align*}
   \frac{d}{dt} x &=  f_a(x), \\
   \frac{d^2}{dt^2} x & = \big( \partial_x f_a(x)\big) \frac{d}{dt} x =
  \big( \partial_x f_a(x)\big)f_a(x)=f_{aa}(x), \\
  \frac{d^3}{dt^3} x & = \big( \partial_x f_{aa}(x)\big) \frac{d}{dt} x =
  \big( \partial_x f_{aa}(x)\big)f_a(x)=f_{aaa}(x), \\
   \cdots &=\cdots
\end{align*}
For an alphabet with \( N\) letters, if \( \lambda_{a}(t) = 1\) for each letter and each \(t\), then \( \alpha_w = (t-t_0)^n/n! \) for any of the   \( N^n\) words \( w \)  with \( n \) letters. In this  case, \eqref{eq:solution}  is the Taylor series for \( x(t) \) with the derivatives of \( x(t) \) written in terms of the \( f_{a}\), \(a\in \A\).

It is also important to note that, in \eqref{eq:solution}, the coefficients \( \alpha_w \) depend on the functions \( \lambda_a \), \( a\in\A\), and are independent of the \( f_a \) in \eqref{pvideter}; on the contrary, the word basis functions \( f_w \) are independent of
the \( \lambda_a \) and change with \( f_a \). This will make it possible to compare later different splitting integrators by expressing them in terms of one common set of word basis functions.

The following two results will be required in the next section.

\begin{prop}\label{prop:alphagroup}
For  any choice of the functions \( \lambda_a\), \(a\in\A\), and any  \( t  \), \( t_0 \), the coefficients \( \alpha_w(t;t_0) \) computed in \eqref{eq:alpha1}--\eqref{eq:alpha2} satisfy the shuffle relations, i.e.\ the element
\( \alpha(t;t_0)\in\RW\) lies in the group \( \G \subset \RW\).
\end{prop}

\begin{prop}\label{prop:chen}Assume that \(t_0<t_1<t_2\), then, for  any choice of the functions \( \lambda_a\), \(a\in\A\), with the notation as above,
\[
\alpha(t_2;t_0) = \alpha(t_1;t_0) \star \alpha(t_2;t_1).
\]
\end{prop}

As an example, for the two-letter word \(ab\), the  proposition yields, since \( \alpha_{\emptyset}(t_1;t_0) = \alpha_{\emptyset}(t_2;t_1) = 1\),
\[
\alpha_{ab}(t_2;t_0)= \alpha_{ab}(t_2;t_1)+\alpha_a(t_1;t_0) \alpha_b(t_2;t_1)+\alpha_{ab}(t_1;t_0),
\]
an equality that may be obtained elementary by writing the left-hand side as a double integral over a triangle and then decomposing the triangle into two smaller triangles and a rectangle.

\subsection{Word series operators}
Real-valued functions \( \chi \)  defined in \(\R^d\) shall be called {\em observables}.
For every letter \( a \), \( D_a \) is the linear  differential operator  that maps the observable \( \chi \)
into the new observable \( D_a \chi \) defined by
\[
D_a \chi (x)=\sum_{i=1}^d f_a^i(x)\frac{\partial}{\partial x_i}\chi(x),\qquad x\in\R^d.
\]
For each word \( w=a_1a_2\cdots a_n \) with more than one letter, we define
the operator \( D_w \) by composing the operators associated with the letters of \(w\):
\[
D_{a_1a_2\dots a_n}=D_{a_1}\circ D_{a_2}\circ\cdots\circ D_{a_n}.
\]
For the empty word the  corresponding operator is the identity: \( D_\emptyset \chi(x) = \chi(x)\). Note that the dependence of the \(D_w\) on the functions \(f_a\) is not incorporated into the notation.
Given \( \delta\in\RW \), we define its {\em word series operator} as the  formal linear differential operator:
\[
D_{\delta}=\sum_{w\in\W}{\delta_w D_w}.
\]
It is trivial to check that convolution product \( \star \) is defined in such a way that it  corresponds to the  composition
of the associated word series operators:
\[
D_\delta \circ D_{\delta^\prime} = D_{\delta \star \delta^\prime},\qquad \delta,\delta^\prime\in\RW.
\]

The differential operators \( D_w\), \( w\in\W \), may also be applied in a componentwise way to vector-valued observables  defined in \( \R^d \).  By considering the application of \(D_w\) to the identity map \( id: x \mapsto x \), \(x\in\R^d\), we find that the word basis function \( f_w \) and the operator \( D_w \) are related through the formula
\(
f_{w}=D_{w} id
\).
By implication, \( W_\delta(x) = D_\delta id(x)\) for \( \delta \in \RW\) and \(x \in \R^d \).

\subsection{Handling word series and word series operators}

The following theorem provides   rules for handling word series  and word series operators. Note the order in which \(\gamma\) and \(\delta\) appear in \eqref{eq:Gcomp}.

\begin{theo}\label{th:rules}
 Let \( \gamma \) be an element of the group \( \G \). Then:
\begin{itemize}
\item  (Composition of a word series and an observable.) For any   (real or vector valued) observable
 \( \chi \),
\begin{equation}\label{eq:Gobs}
\chi\big( W_{\gamma}(x) \big)=D_{\gamma}\chi(x),
\end{equation}
\item (Composition of word series.) For every \( \delta\in \RW\), we have
\begin{equation}\label{eq:Gcomp}
W_{\delta}\big(W_{\gamma}(x)\big)=W_{\gamma\star\delta}(x).
\end{equation}
\end{itemize}
\end{theo}

It is  important to emphasize that  the hypothesis \( \gamma\in\G \) is essential for the result to hold; the conclusions are not true if \( \gamma\in\RW \) does not belong to the group.
According to Proposition~\ref{prop:alphagroup}, the coefficients \( \alpha_w(t;t_0)\) may play the role of
\(\gamma \) in the theorem. This is the key to the analysis of splitting integrators, as we show in Section 3.

\subsection{References and discussion}
\label{sec:discussion}

The material in Section~\ref{sec:words} is connected to several algebraic theories, even though, for the benefit of more applied readers, those connections have been downplayed  in our exposition. The vector space \(\RW\) is the dual of the shuffle Hopf algebra and the group \(\G\) is the group of characters of such algebra, see \cite{words} and its references. The monograph \cite{reut} contains many relevant results on the combinatorics of words.

Series indexed by the words of an alphabet were introduced and studied extensively by Chen, see e.g. \cite{chen}. Sometimes the series are presented as combining words themselves, i.e. they are of the form \(\sum_w \delta_w w\) with \( \delta\in\RW\) (Chen series). In other applications, notably in control theory \cite{f}, the series combine differential operators as in
our \(D_\delta= \sum_w \delta_w D_w\) above. Word series \cite{anderfocm}, \cite{part2}, \cite{part3}, \cite{words} while essentially equivalent to Chen series are series of mappings and therefore, in numerical analysis, they may be used in the same way as B-series. Word series may also be used to study analytically dynamical systems: \cite{words}, \cite{juanluis}, \cite{k}. Chen series also play an important role in Lyons rough  path theory, see e.g.\ \cite{b}.

Each series basis function \(f_w\) may be decomposed as a sum of elementary differentials \cite{words}. After such a decomposition each word series becomes a B-series; the B-series has a term for each coloured rooted tree. Since there are far more coloured rooted trees with \(n\) vertices than words with \(n\) letters, the B-series format is less compact.
An additional advantage of word series over B-series is the simplicity of the operation \(\star\); the rule for composing B-series is substantially more complicated. On the other hand word series have a more limited scope than B-series:
not all B-series may be rewritten as word series. Splitting integrators may be described by word series, but that is not the case for Runge-Kutta algorithms
or additive Runge-Kutta algorithms \cite{aderito}.

A proof Theorem~\ref{th:rules} may be seen in \cite{words}.
The fact that iterated integrals satisfy the shuffle relations (Proposition~\ref{prop:alphagroup}) was first noted by Ree \cite{ree}. Proposition~\ref{prop:chen} is due to Chen \cite{chen}; in view of \eqref{eq:solution} and \eqref{eq:Gcomp}, the result expresses in terms of words the composition rule for solution operators \(\phi_{t_2;t_0}= \phi_{t_2;t_1}\circ\phi_{t_1;t_0} \).

\section{Composing exact solutions with the help of word series}

Theorem \ref{th:rules} leads  to a technique to represent the local error of splitting integrators both for deterministic and Stratonovich equations. Even though the idea is completely general, for notational convenience we shall present it by means of a very simple (deterministic) example. Consider the particular case where in the system
\eqref{pvideter}  the alphabet consists of three letters \(\A=\{a,b,c\}\), i.e.
\[
\frac{d}{dt}x = \lambda_a(t)f_a(x)+\lambda_b(t)f_b(x)+\lambda_c(t)f_c(x),
\]
and denote by \( \phi_{t,t_0}: \R^d\to\R^d\) the solution mapping, i.e.\ the mapping
such that, for each \( x_0\), \( \phi_{t,t_0}(x_0)\) is the value at \( t \) of the solution with initial condition \( x(t_0)=x_0\).
Assume that the split systems
\[
\frac{d}{dt}x = \lambda_a(t)f_a(x)+\lambda_b(t)f_b(x),\quad \frac{d}{dt}x = \lambda_c(t)f_c(x),
\]
may be integrated analytically and denote by \( \phi^{(1)}_{t,t_0}: \R^d\to\R^d\) and  \( \phi^{(2)}_{t,t_0}: \R^d\to\R^d\) their solution mappings.
The simplest splitting integrator  advances the numerical solution from \( t_0 \) to \( t_0+h\), \(h>0\), by means of the mapping
\[
\widetilde\phi_{t_0+h,t_0} = \phi^{(2)}_{t_0+h,t_0}\circ  \phi^{(1)}_{t_0+h,t_0}.
\]

From \eqref{eq:solution} we have the word series representation (we write \( \alpha_w \) instead of \( \alpha_w(t_0+h;t_0) \))
\begin{eqnarray*}
\phi_{t_0+h,t_0}(x_0) &=& W_{\alpha(t_0+h;t_0)}(x_0)\\
 &= &x_0 + \alpha_a f_a(x_0)+\alpha_b f_b(x_0)+\alpha_cf_c(x_0)\\&&\qquad\qquad+\alpha_{aa}f_{aa}(x_0)+\alpha_{ab}f_{ab}(x_0)+\alpha_{ac}f_{ac}(x_0)+\cdots
\end{eqnarray*}
(note that for simplicity only three of the nine terms with two letters have been displayed).
For the first split system, using still the alphabet \(\{a,b,c\}\) and including all words with two letters,
\begin{eqnarray*}
\phi^{(1)}_{t_0+h,t_0}(x_0) &=& W_{\alpha^{(1)}(t_0+h;t_0)}(x_0)\\
 &= &x_0 + \alpha_a f_a(x_0)+\alpha_b f_b(x_0)\\&&\qquad\qquad+\alpha_{aa}f_{aa}(x_0)+\alpha_{ab}f_{ab}(x_0)\\
 &&\qquad\qquad+\alpha_{ba}f_{ba}(x_0)+\alpha_{bb}f_{bb}(x_0)+\dots;
\end{eqnarray*}
when computing the coefficients \( \alpha^{(1)}_w \) by means of \eqref{eq:alpha1}--\eqref{eq:alpha2}
we have to take
\( \lambda_c(t) =0 \), so that \( \alpha^{(1)}_w= \alpha_w\) if \( w\) does not contain the letter \( c\) and
\( \alpha^{(1)}_w=0\) otherwise.
Similarly,
\begin{eqnarray*}
\phi^{(2)}_{t_0+h,t_0}(x_0) &=& W_{\alpha^{(2)}(t_0+h;t_0)}(x_0)\\
 &= &x_0 + \alpha_c f_c(x_0)+\alpha_{cc}f_{cc}(x_0)+\cdots,
\end{eqnarray*}
where the dots stand for words with three or more letters.
Now, after invoking Proposition~\ref{prop:alphagroup} and \eqref{eq:Gcomp}, we find
\begin{equation}\label{eq:seriesalphatilde}
  \widetilde\phi_{t_0+h,t_0}(x_0) = W_{\widetilde\alpha(t_0+h;t_0)}(x_0),
\end{equation}
with
\[
\widetilde\alpha(t_0+h;t_0) = \alpha^{(1)}(t_0+h;t_0)\star \alpha^{(2)}(t_0+h;t_0)\in\G.
\]
By using the definition of the convolution product \(\star\), we compute
\begin{eqnarray*}
\widetilde\phi_{t_0+h,t_0}(x_0) &=& x_0 + \alpha_a f_a(x_0)+\alpha_b f_b(x_0)+\alpha_cf_c(x_0)\\
&&   +\alpha_{aa}f_{aa}(x_0)+\alpha_{ab}f_{ab}(x_0)+\alpha_{a}\alpha_{c}f_{ac}(x_0)\\
&&   +\alpha_{ba}f_{ba}(x_0)+\alpha_{bb}f_{bb}(x_0)+\alpha_{b}\alpha_{c}f_{bc}(x_0)+\alpha_{cc}f_{cc}(x_0)+\cdots
\end{eqnarray*}
It is extremely easy to find the coefficients in the last expansion. If \( w\) is  a concatenation \(w^\prime w^{\prime\prime}\), where the (possibly empty) word \(w^\prime\) does not include the letter \(c\) and the (possibly empty) word \(w^{\prime\prime}\) does not include the letters \(a\) or \(b\), then \(\widetilde\alpha_w= \alpha_{w^\prime}\alpha_{w^{\prime\prime}}\); if \( w\) is not a concatenation of that form, then the coefficient is \(0\).

The fact that the expansion of the integrator mapping \( \widetilde \phi\) and the solution mapping
\(\phi\) agree for words with \( < 2\) letters implies, via \eqref{eq:alphabound}, that both differ by
\(\mathcal{O}(h^2)\), i.e.\ that the integrator is consistent.
The local error  may be expanded as a word series
\[
\widetilde\phi_{t_0+h,t_0}(x_0) -\phi_{t_0+h,t_0}(x_0)=
W_{\delta(t_0,h)}(x_0)\] with
\[
\delta(t_0,h) =\alpha^{(1)}(t_0+h;t_0)\star \alpha^{(2)}(t_0+h;t_0)-\alpha(t_0+h,t_0).
\]
In particular, the leading \(\mathcal{O}(h^2)\) term, corresponding to two-letter words, is given by:
\[
    (\alpha_a\alpha_c-\alpha_{ac}) f_{ac}(x_0)
     + (\alpha_b\alpha_c-\alpha_{bc}) f_{bc}(x_0)
    - \alpha_{ca}f_{ca}(x_0) - \alpha_{cb}f_{cb}(x_0).
\]

In some circumstances (for instance when studying conservation of energy or other invariants of motion) it is of interest to look at the error in an observable \(\chi\) after a single step:
\[
 \chi\Big(\widetilde\phi_{t_0+h,t_0}(x_0)
 \Big)-\chi\Big(\phi_{t_0+h,t_0}(x_0)\Big),
\]
Expansions of errors of this kind are easily derived with the help of \eqref{eq:Gobs}. In our example, we  may write, {\em without any additional computation},
\begin{eqnarray*}
  \chi\Big(\widetilde\phi_{t_0+h,t_0}(x_0)
 \Big)-\chi\Big(\phi_{t_0+h,t_0}(x_0)\Big)&=&   (\alpha_a\alpha_c-\alpha_{ac}) D_{ac}\chi(x_0)\\
     && + (\alpha_b\alpha_c-\alpha_{bc}) D_{bc}\chi(x_0)\\
    &&- \alpha_{ca}D_{ca}\chi(x_0)- \alpha_{cb}D_{cb}\chi(x_0)+\cdots,
\end{eqnarray*}

For this simple example the results presented here could have been found easily by elementary means. However,
as pointed out above, the word series technique works for arbitrary  splitting coefficients leading to high-order algorithms and arbitrary ways of splitting the right-hand side of \eqref{pvideter} into two or more parts.

\section{Splitting methods for Stochastic Differential Equations}

In this section we show how word series may be used to analyze local errors of splitting integrators for SDE.

\subsection{Expanding the true solution}
Consider the  \(d\)-dimensional   Stratonovich SDE,
\begin{equation}\label{eq:SDE}
dx=\sum_{a\in\A_{det}}f_{a}(x)\,dt+\sum_{A\in\A_{stoch}}f_{A}(x)
\circ dB_A(t)
\end{equation}
where \( \A_{det}\) and \(\A_{stoch}\) are finite sets without common elements and the  \(B_{A}(t)\), \( A\in \A_{stoch}\), are  independent scalar Wiener processes defined on the same filtered probability  space.  We shall use the material above with the alphabet \(\A = \A_{det} \cup  \A_{stoch}\).
The letters in \( \A_{det}\)  (respectively in \(\A_{stoch}\))   are called  {\em deterministic} (respectively  {\em stochastic}).  The {\em weight}     $\| w \|$ of the letter \( w \) is defined as the number of deterministic letters of   \( w\) plus a half of the number of stochastic letters. The weight thus takes values in the set \( \Nh =\{ 0, 1/2, 1, 3/2,\dots\}\). Note that, if the \( w_j\) are the words resulting from shuffling \( w\) and \( w^\prime\) as in \eqref{eq:sh}, then, for each \( j\), \( \|w_j\| = \|w\|+\| w^\prime\|\). Also when two words are concatenated the weight of the result is the sum of the weights of the factors.

Since  Stratonovich integrals follow the rules of ordinary calculus, from \eqref{eq:solution} we conclude that the solution of \eqref{eq:SDE} with initial condition \( x(t_0) = x_0\) has the expansion, \( t > t_0 \geq 0\),
\begin{equation}\label{eq:seriesJ}
  x(t) = W_{J(t;t_0)}(x_0),
\end{equation}
where the
\( J_{w}(t;t_0)\) are the well-known Stratonovich iterated integrals (\( wa\) and \(wA\) are  the words obtained by appending the letter \(a \) or \( A\) at the end of \(w\)):
\begin{align*}
    J_{\emptyset}(t;t_0)&=1,\\
    J_{a}(t;t_0)&=\int_{t_0}^{t}\,{ds}=t-t_0,\qquad a\in\A_{det},\\
    J_{A}(t;t_0)&=\int_{t_0}^{t}\circ dB_{A}(s)=B_A(t_1)-B_A(t_0),\qquad A\in\A_{stoch},\\
    J_{wa}(t,t_0)&=\int_{t_0}^{t}{J_{w}(s;t_0)\, ds},\qquad a\in\A_{det},\\
    J_{wA}(t;t_0)&=\int_{t_0}^{t}J_w(s;t_0)\,{\circ dB_A(s)},\qquad A\in\A_{stoch}.
\end{align*}

The expansion \eqref{eq:seriesJ} of course coincides with the familiar Stratonovich-Taylor expansion (see e.g.\ \cite[Chapter 5]{kloeden}).

The following result summarizes some  properties of the \( J_w(t;t_0)\) which will be required later. The first item expresses the shuffle relations of iterated integrals, see also Proposition \ref{prop:alphagroup}. The second, third and fifth item are well known. The fourth is a trivial consequence of the second and third.

\begin{prop}\label{prop:J}The iterated srochastic Stratonovich integrals \( J_w(t;t_0)\) possess the following properties.

\begin{itemize}
\item \( J(t;t_0)\in\G\).
\item The joint distribution of any finite subfamily of the family of random variables \( \{ h^{-\|w\|} J_w(t_0+h;t_0)\}_{w\in\W}\) is independent
of \(t_0\geq 0\) and \( h>0\).
\item \( \E\mid J_w(t_0+h;t_0)\mid^p<\infty\), for each \( w\in\W\), \(t_0\geq 0\), \( h>0 \) and \(p\in [0,\infty)\).
\item For each \( w\in\W\) and any finite \( p \geq 1\), the  (\(t_0\)-independent) \(L^p\) norm of the random variable \( J_w(t_0+h;t_0)\) is \(\mathcal{O}(h^{\|w\|})\), as \( h\downarrow 0\).
\item
\( \E\ J_w(t_0+h;t_0) = 0\) whenever \( \|w\|\) is not an integer.
\end{itemize}
\end{prop}

In view of   the Proposition, when the word series in \eqref{eq:seriesJ} is rewritten as
\[
x(t) =\sum_{n\in \Nh}\:\: \sum_{\| w\| = n} J_w(t;t_0)f_w(x_0),
\]
for each \( n \in \Nh\), the term in the inner sum  is \( \mathcal{O}((t-t_0)^{n})\) in any \( L^p\) norm, \(p<\infty\). This should be compared with the deterministic case, where, as we saw above,  the bound \eqref{eq:alphabound} leads to grading  the expansion
\eqref{eq:solution}  by the number of letters of the words.

We shall need below the following auxiliary result (\( \Pi\) denotes of course a product):
\begin{lemma}\label{lemma:vanishing}
Assume that \( w_1, \dots, w_\ell,\) are words with \( \sum_j \|w_j\|\notin \N\). Then, for each \(t_0\geq 0\) and \(h>0\),
\[
\E \big(\Pi_j J_{w_j}(t_0+h;t_0)\big) = 0.
\]
\end{lemma}
\begin{proof} By using repeatedly the shuffle relations \eqref{eq:sh}--\eqref{eq:sh2}, the product of iterated integrals  may be rewritten as a sum of iterated integrals corresponding to the words \( w^\prime_i\) resulting from shuffling the \( w_j\), \( j = 1,\dots,\ell\). As noted above each \( w^\prime_i\) has the non-integer weight \( \sum_j \|w_j\|\)and we may use the last item of Proposition~\ref{prop:J}.
\end{proof}

The idea of the proof (i.e\ the use of the shuffle relations to rewrite products of iterated integrals as as sums) has been used in \cite{gaines} as a means to evaluate the moments of iterated stochastics integrals. An instance of the shuffle relations for iterated stochastic integrals is presented in Proposition 5.2.10 of  \cite{kloeden}; this well-known monograph does not relate the formula presented there to the algebra of word shuffles. A number of recent papers have also exploited the connection between the Stratonovich calculus and the shuffle Hopf algebra, see e.g.\ \cite{kur} and its references.

\subsection{Expanding the numerical solution}

In a splitting integrator, a time-step \( t_0\rightarrow t_0+h\), \( h>0\), is performed by applying a mapping
\(
\widetilde\phi_{t_0+h,t_0}
\)
defined as a composition of several solution mappings \[ \phi^{(i)}_{t_0+c_ih,t_0+d_ih},\qquad i=1,\dots,I,\] corresponding to SDEs resulting from splitting the right-hand side of
\eqref{eq:SDE}. The \( c_i\) and \(d_i\) are real constants associated with the particular integrator. By proceeding as in the deterministic case, the use of the operation \(\star\) leads to a word-series representation (cf.\ \eqref{eq:seriesalphatilde}),
\[
\widetilde\phi_{t_0+h,t_0}(x_0) = W_{\widetilde J(t_0+h;t_0)}(x_0),\qquad i = 1,\dots,I,
\]
where, for each nonempty \( w\in\W\), \( \widetilde J_w(t_0+h;t_0)\) is either zero or a sum of products of iterated Stratonovich integrals corresponding to words whose concatenation is \( w\). Therefore, in each product, the iterated integrals being multiplied  correspond to words  whose weights add up to  \( \|w\| \).

\begin{prop}\label{prop:Jtilde} The coefficients \( \widetilde J_w(t_0+h;t_0)\), \(w\in\W\), associated with a splitting integrator possess the properties  of the exact values \( \widetilde J_w(t_0+h;t_0)\) listed in Proposition~\ref{prop:J}.
\end{prop}
\begin{proof} The first four items of this proposition are   consequences of Proposition~\ref{prop:J} and the representation of each \( \widetilde J_w(t_0+h;t_0)\), \( w\neq \emptyset\), as a sum of products of
iterated  integrals. For the last item, in view of the linearity of the expectation, it is enough to prove that, for any \(t_j < t_j^*\),
\[
\sum_j \| w_j\|\notin\N \quad\Rightarrow\quad
\E\big(\prod_j J_{w_j}(t_j^*;t_j)\big)= 0.\]
Furthermore, we may assume that we are in the particular case where any two intervals \( (t_j,t_j^*)\subset \R\) are either disjoint or equal to each other; the general situation  may be reduced to the particular case by  decomposing with the help of Proposition~\ref{prop:chen}.
  Under this assumption, let us group together the iterated integrals sharing the same \(  (t_j,t_j^*) \) and write
\[
\prod_j J_{w_j}(t_j^*;t_j) =\prod_k \prod_{j\in I_k} J_{w_j}(t_k^*;t_k);
\]
here, as \( k\) varies, any two intervals \((t_k,t_k^*)\subset\R\) are disjoint, and, for each value of \( k\), the set
\(I_k\) comprises the indices \( j\) for which  \((t_j^*;t_j)\) coincides with \((t_k^*;t_k)\). Now, by  independence,
\[
\E \big(\prod_j J_{w_j}(t_j^*;t_j)\big)  = \prod_k \E\big(\prod_{j\in I_k} J_{w_j}(t_k^*;t_k)\big),
\]
and the proof will be completed if we show that there is at least a value of \(k\) for which
\[
\E\big(\prod_{j\in I_k} J_{w_j}(t_k^*;t_k)\big) = 0.
\]
Since
\[
 \sum_{k} \sum_{j\in I_k} \| w_j\| = \sum_j \|w_j\| \notin \N,
\]
at least one of the inner sums  is not an integer and we may apply Lemma~\ref{lemma:vanishing}.
\end{proof}

\subsection{The local error}

The preparations above have proved the main result of this article:

\begin{theo}\label{th:main} For a splitting integrator as above, the local error possesses a word series expansion
\begin{equation}\label{eq:main1}
\widetilde\phi_{t_0+h,t_0}(x_0)-\phi_{t_0+h,t_0}(x_0) = W_{\delta(t_0,h)}(x_0) = \sum_{n\in \Nh}\:\: \sum_{\| w\| = n} \delta_w(t_0,h) f_w(x_0),
\end{equation}
with coefficients
\[
\delta_w(t_0,h) = \widetilde J_w(t_0+h;t_0)-J_w(t_0+h;t_0),\qquad w\in \W,
\]
that, in any \( L^p\) norm, \(1\leq p< \infty\), satisfy, uniformly in \(t_0\geq 0\),
\[
\| \delta_w(t_0,h)\|_p =\mathcal{O}(h^{\|w\|}), \qquad h\downarrow 0.
\]

In addition, for each observable \( \chi\), conditional on \( x_0\),
\begin{equation}\label{eq:main2}
\E\chi\big( \widetilde\phi_{t_0+h,t_0}(x_0)\big)-\E\chi\big(\phi_{t_0+h,t_0}(x_0)\big) = \sum_{n\in \N}\:\: \sum_{\| w\| = n} \big(\E\delta_w(t_0,h)\big) D_w \chi (x_0).
\end{equation}
\end{theo}

The theorem implies that the {\em  strong order conditions}
\begin{equation}\label{eq:oc1}
\widetilde J_w(t_0+h;t_0)=J_w(t_0+h;t_0),\qquad \|w\| = 0,1/2,1,\dots, \mu,\quad \mu\in\Nh,
\end{equation}
ensure that the series in \eqref{eq:main1} only comprises terms of size \( \mathcal{O}(h^{\mu+1/2})\). In fact, under suitable assumptions on \eqref{eq:SDE}, the fulfillment of the order conditions ensures that the local error possesses an \( \mathcal{O}(h^{\mu+1/2})\)  bound (see the Appendix).

It should be pointed out that, since both \( J(t_0+h;t_0) \) and \( \widetilde J(t_0+h;t_0) \) satisfy the shuffle relations, the conditions in \eqref{eq:oc1} corresponding to different words are not independent from one another. For instance, from the shuffle
\( a\sh a = 2 aa\), \( a\in\A\), we may write
\[
\big(J_a(t_0+h;t_0)\big)^2 = 2 J_{aa}(t_0+h;t_0),\quad  \big( \widetilde J_a(t_0+h;t_0) \big)^2 = 2 \widetilde J_{aa}(t_0+h;t_0),
\]
and therefore the order condition for the word \( aa\) is fulfilled if and only if the same happens for \( a\). Lyndon words \cite{reut} may be used to identify subsets of independent order conditions (cf. \cite{phil}) but we shall not concern ourselves with such an investigation.

If, for a given alphabet \(\A\) and given coefficients  \( \widetilde J_w(t_0+h;t_0) \), one demands
that the series in \eqref{eq:main1} only comprises terms of size \( \mathcal{O}(h^{\mu+1/2})\)
for {\em all possible choices} of the vector fields \( f_a\), \(f_A\), then the conditions \eqref{eq:oc1}  are not only sufficient but also {\em necessary}. This happens because, as it is easy to show, in such a scenario, the word basis functions are mutually independent. However this consideration is not of much practical value;  splitting integrators are useful because they are adapted to the specific structure of the problem being solved and therefore one is interested in the behavior for individual problems not in catering for all possible choices of \( f_a\), \(f_A\). The best way to deal with specific problems is to write down, up to the desired order, the word series expansions of
the true and numerical solutions and compare them after taking into account the shuffle relations and the specific expressions of the word basis functions; this will be illustrated in the next section.
For instance, if,
for the problem at hand, a word basis function \(f_w\) vanishes identically, then it is clearly not necessary to impose the associated order condition in \eqref{eq:oc1}.

Similar considerations apply to the {\em weak order conditions}
\begin{equation}\label{eq:oc2}
\E\widetilde J_w(t_0+h;t_0)=\E J_w(t_0+h;t_0),\qquad \|w\| = 0,1,2,\dots, \nu,\quad \nu\in\N,
\end{equation}
which ensure that the series in \eqref{eq:main2} only comprises terms of size \( \mathcal{O}(h^{\nu+1})\).

The conditions \eqref{eq:oc1}--\eqref{eq:oc2} are similar to those found in \cite{burrage} for stochastic Runge-Kutta integrators (however \cite{burrage} only shows that a condition corresponding to \eqref{eq:oc2} implies
that the expectation of the local error is \( \mathcal{O}(h^{\nu+1})\); arbitrary observables \(\chi\) are not considered there).

\section{Application to Langevin dynamics}
We shall illustrate the application of the foregoing material by considering the Langevin  equations
\begin{eqnarray*}
 dq &  = & M^{-1}p\,dt \\
 dp &  = & F(q)\,dt-\gamma p\,dt+\sigma M^{1/2} dB(t),
\end{eqnarray*}
where \(M\)  is the  \(d\times d\) diagonal mass matrix with diagonal entries \(m_i>0\),  \(\gamma >0\) is the friction coefficient, \(\sigma\) governs the fluctuation due to noise, \(B\) is a  \(d\)-dimensional Wiener process, and  the force \(F\) originates from a potential, i.e. \(F = -\nabla V\) for a suitable scalar-valued function \(V\). Since the noise is additive there is no distinction between the Stratonovich and Ito interpretations.

\subsection{Splitting the Langevin dynamics}

After setting \(x=(q,p)\in\R^d\times\R^d\), the equations are the particular instance of \eqref{eq:SDE} given by
\begin{equation}
dx(t)=f_{a}(x)dt+f_{b}(x)dt+f_{c}(x)dt+\sum_{i=1}^d f_{A_i}(x)\circ dB_i(t)
\label{eq:SDElangevin}
\end{equation}
with
\[
f_{a}(q,p)=(M^{-1}p,0),\quad f_{b}(q,p)=(0,F(q)), \quad f_{c}(q,p)=(0,-\gamma p),
\]
and, for \(i=1,\dots,d\),
\[
f_{A_i}(q,p)=(0,\sigma\sqrt{m_i}e_i),
\]
where \( e_i\) is the \(i\)-th unit vector in \( \R^d\). The deterministic letters \(a\), \(b\) and \(c\) are respectively  associated with inertia, potential forces and friction; as it will become apparent below the word basis functions \( f_w\), \( w\in\W\) also have clear physical meaning.

The system \eqref{eq:SDElangevin} is split into three parts corresponding to  \(\{f_a\}\), \(\{f_b\}\) and \(\{f_c\), \(f_{A_1}\), \dots, \(f_{A_d}\}\).\footnote{The splitting considered here is not the only meaningful way to split the Langevin equations; a Hamiltonian/Orstein-Uhlenbeck splitting is considered  in e.g.\ \cite{nawaf}. See also \cite{Leim3}.}
Each of the three split systems may be integrated explicitly. With a terminology common in molecular dynamcics, the solution of the first is a \lq drift\rq\ in position, \( q \mapsto q+(t-t_0)M^{-1}p\) (\(p\) remains constant). The solution of the
second  is a  \lq  kick\rq\ in momentum \(p\mapsto p+(t-t_0)F(q)\) (\(q\) remains constant). The third split system defines an Ornstein-Uhlenbeck process in \(p\). Leimkuhler and Matthews \cite{Leim}, \cite{Leim2} use the letters A, B and O to refer to these split systems and the acronym ABOBA for the Strang-like algorithm
\[
\widetilde\phi^{ABOBA}=\phi^A_{t_0+h;t_0+h/2}\circ\phi^B_{t_0+h;t_0+h/2}\circ \phi^O_{t_0+h;t_0}\circ\phi^B_{t_0+h/2;t_0}\circ\phi^A_{t_0+h/2;t_0}.
\]
With the help of an analysis of the large friction limit and numerical experiments, these authors find that the very similar BAOAB algorithm
\[
\widetilde\phi^{BAOAB}=\phi^B_{t_0+h;t_0+h/2}\circ\phi^A_{t_0+h;t_0+h/2}\circ \phi^O_{t_0+h;t_0}\circ\phi^A_{t_0+h/2;t_0}\circ\phi^B_{t_0+h/2;t_0}
\]
substantially improves on ABOBA. In this section we analyze by means of word series the local error of both algorithms. Our findings complement (rather than duplicate) those in \cite{Leim},  \cite{Leim2},  \cite{Leim3}.

\begin{table}[t]
\begin{center}
\begin{tabular}{c|c|c|c|c}
\( \|w\|\) & \(w\) & \(\widetilde J_w^{ABOBA}\)  & \(\widetilde J_w^{BAOAB}\)  & Exact?\\\hline
0& \(\emptyset\)  & \(1 \)  &  \(1 \)     &      \checkmark\\ \hline
\(1/2\)&   \(A_j\)    &  \(J_{A_j}\)  & \(J_{A_j}\)  &      \checkmark\\\hline
\(1 \)  & \(a,b,c\)   &   \(h \)      &   \(h \)     &      \checkmark\\\hline
\(3/2\) & \(A_ia\)    & \(hJ_{A_i}/2\) & \(hJ_{A_i}/2\) & \\
        &\(A_ic\)     & \(J_{A_ic}\) & \(J_{A_ic}\) &  \checkmark       \\\hline
2& \(ab,ba,bc,ca,cc\) & \(h^2/2\)  & \(h^2/2\) &  \checkmark  \\\hline
5/2&\(A_iab\)   &0   &\(h^2J_{A_i}/4\)&  \\
&\(A_ica\) &  \(hJ_{A_ic}/2\)  & \(hJ_{A_ic}/2\) &  \\
&\(A_icc\) &  \(J_{A_icc}\) & \(J_{A_icc}\) & \checkmark \\
\end{tabular}
\end{center}
\caption{Coefficients of the splitting methods ABOBA and BAOBA for words \(w\) with weight \(\|w\|<3\) and nonvanishing basis function \(f_w\). A check mark signals agreement with the exact \(J_w\). All iterated stochastic integrals have domain \((t_0+h;t_0)\).}
\label{table1}
\end{table}

\subsection{The word basis functions}

The structure of the Langevin equations implies that many word basis functions are identically zero.
The vector fields \(f_a\), \( f_b\), \(f_c\), and \(f_{A_i}\)  have many null components and additional simplifications are due to  \(f_{A_i}\) being constant, \(f_a\) and \(f_c\) being linear in \( p\) and independent of \(q\), and \(f_b\) being independent of \(p\). In particular, the relation \( f_{ba}(q,p) = (M^{-1}F(q),0)\)
shows that \(f_{ba}\) is a function of \(q\) alone and, since the \(q\) components of \(f_c\)  and \(f_{A_i}\) vanish,
we have, in view of  \eqref{eq:wbs},
\begin{equation}\label{eq:basisvanish}
f_{cba}(q,p)= 0,\qquad f_{A_iba}(q,p)= 0,\: i = 1,\dots,d,
\end{equation}
for each \(q\) and \(p\).
Physically, \eqref{eq:basisvanish} means that  the value \( M^{-1}F(q) \)  of  the acceleration created by the potential forces would not be affected if noise or friction changed instantaneously the
momentum of the system.
On the other hand, in general,
\begin{equation}\label{eq:basisvanish2}
f_{cab}(q,p)\neq 0,\qquad f_{A_iab}(q,p)\neq 0,\quad i = 1,\dots,d.
\end{equation}
The second block of \( f_{ab}(q,p) = (0,\partial_q F(q)M^{-1}p)= (0,(d/dt) F(q) )\) is the contribution  to \((d^2/dt^2)p\) that arises from the potential forces. This contribution is a function of \(q\) {\em and} \(p\) and its value would be affected if friction or noise changed instantaneously the momentum.  It is also useful to note at this point that, according to \eqref{eq:wbs}, if \(f_w\) vanishes identically, then the same is true for all words of the form \(w^\prime w\), i.e.\ for all words that have \( w\) as a suffix.
Table~\ref{table1} lists the words \( w\) with weight \( < 3\) and nonvanishing basis function.

\subsection{Coefficients}
Once the relevant word basis functions have been identified, we proceed to find the coefficients. Let us begin with  ABOBA. From the definition of the operation \(\star\), it is clear that,
if \( w\)  is not of the form \( a^kb^\ell w^\prime b^ma^n\), with \(k, \ell, m, n\) nonnegative integers and \(w^\prime\) a word not including the letters \(a\) or \(b\), then
\(
\widetilde J^{ABOBA}_w = 0.
\)
For a word that may be written in that form in a unique way (e.g.\ \(abccba\)),  the value of \(\widetilde J^{ABOBA}_w\) is
\begin{align*}
    &J_{a^k}(t_0+h/2;t_0)J_{b^\ell}(t_0+h/2;t_0)J_{w^\prime}(t_0+h;t_0)  \\
    &\qquad\qquad\qquad\qquad\times J_{b^m}(t_0+h;t_0+h/2)J_{a^n}(t_0+h;t_0+h/2)
\end{align*}
or, from \eqref{eq:taylor},
\[
 = \frac{1}{k!\ell!m!n!}\left(\frac{h}{2}\right)^{k+\ell+m+n} J_{w^\prime}.
\]
For a word that may be written in  the form \( a^kb^\ell w^\prime b^ma^n\) in several ways, we sum over all possible ways (e.g.\ for \(aa\), we have  \(\ell=m = 0\), \(w^\prime = \emptyset\), and three possibilities, \( (k,n) = (2,0) \), \( (k,n) = (1,1) \), \( (k,n) = (0,2) \) leading to a coefficient \( (1/2)(h/2)^2 + (h/2)^2 + (1/2)(h/2)^2 = h^2/2\)). Similar considerations, with the roles of \(a\) and \(b\) interchanged apply to the alternative BAOAB method. It  now takes next to no time to find the coefficients in the third and fourth columns of the table.

\subsection{Comparing the algorithms}
At this point, we are  in a position to compare the algorithms.
 Since at the words \(A_ia\), \(i=1,\dots,d\), both methods  are in error, for both of them,   the local error expansion in \eqref{eq:main1} begins with \(\mathcal{O}(h^{3/2})\) terms. Furthermore ABOBA and BAOAB share the same coefficient values \( \widetilde J_w\) at the leading (i.e.\ \(\mathcal{O}(h^{3/2})\)) order and also at the next order
(corresponding to words of weight 2). In fact, for the words that feature in the table, the only difference between both integrators corresponds to the words \(A_iab\), \(i=1,\dots,d\). For these, the exact solution has coefficient
\[
J_{A_iab} \sim \mathcal{N}\Big(0,\frac{h^5}{20}\Big),
\]
BAOAB has
\[
\widetilde J^{BAOAB}_{A_iab}= \frac{h^2}{4}J_{A_i} \sim \mathcal{N}\Big(0,\frac{h^5}{16}\Big),
\]
while, as noted above,
\[
\widetilde J^{ABOBA}_{A_iab}= 0,
\]
due to the pattern \(ab\) after the stochastic letter.
The joint distribution of \(J_{A_iab}\) and  \(\widetilde J^{BAOAB}_{A_iab}\) is Gaussian with
covariance \(h^5/24\) and therefore the correlation between both variables is
\[
\frac{h^5/24}{\sqrt{h^5/20}\:\sqrt{h^5/16}} = \frac{\sqrt{5}}{3}\approx 0.74,
\]
while \(J_{A_iab}\) and  \(\widetilde J^{ABOBA}_{A_iab}\) are obviously uncorrelated.
Thus, for this word, ABOBA provides a very poor approximation to the exact coefficient. Due to the symmetric role played by the letters \(a\) and \(b\) in the algebra of words, for  \(A_iba\), it is BAOAB that has an identically zero coefficient. However this is irrelevant for the present discussion because, for that word, the basis function vanishes
as noted in \eqref{eq:basisvanish}.

Cases where \(f_w\neq 0\), \(\widetilde J_w^{ABOBA}=0 \), but \(\widetilde J_w^{BAOAB}\) provides a nontrivial approximation to
\( J_w\)  occur for higher values of the weight. For the deterministic word \(cba\), \( \widetilde J^{BAOAB}_{cab} = h^3/4\) and \(\widetilde J^{ABOBA}_{A_icab} = 0\) (the correct value is \(h^3/6\)).
For \(A_icab\) with weight \( 7/2\), the exact solution has
\[
J_{A_icab} \sim \mathcal{N}\Big(0,\frac{h^7}{252}\Big),
\]
while
\[
\widetilde J^{BAOAB}_{A_icab} \sim \mathcal{N}\Big(0,\frac{h^7}{148}\Big),
\]
and, again due to the \(ab\) pattern,
\[
\widetilde J^{ABOBA}_{A_icab} = 0.
\]
Now the correlation between the BAOAB coefficient and the true value is \(\sqrt{21}/5\approx 0.91\).

Why does ABOBA provide poor approximations for words like \(A_iba\), \(cba\), \(A_icba\)? By looking at the physical meaning of the corresponding word basis functions (see e.g.\ the discussion of \eqref{eq:basisvanish2} presented above),
we see that the {\em above shortcomings of ABOBA stem from the following algorithmic source.} In any given time step, ABOBA uses the {\em same} value of \( F\) in both kicks (\(q\) is not updated between those kicks) and, furthermore, that common value of \(F\) only depends on the values of \(q\) and \(p\) at the beginning of the step. Thus, over the whole step, the momentum increment \(hF\) due to the potential forces does not \lq see\rq\ the presence of friction or noise in the current step. On the contrary, in BAOAB the change in \(p\) at substep O (friction and noise) causes that the  kicking force varies from the first kick to the second.\footnote{Note that BAOBA reuses in the first kick of the next step the value of \(F(q)\) employed in the second kick of the present step, so that both ABOBA and BAOAB use twice each evaluation of the potential force.}

\section{Further developments}

We have presented a systematic method, based on word series, for writing down expansions of strong and weak local errors of splitting integrators for Stratonovich SDEs. The method has been illustrated with a comparison between two related algorithms for the Langevin equations. The material may be adapted to study Ito equations, where the quasishuffle algebra replaces the shuffle algebra used here.

In the deterministic case, word series may also be applied to the computation of modified equations of integrators as in \cite{words}. Similarly the word series approach may also be extended to investigate modified equations for Ito or Stratonovich SDEs. In addition word series may be helpful in finding invariant densities of numerical algorithms. These developments will be dealt with in future work.

\section*{Appendix: error bounds}

In what follows the determistic vector fields \(f_a\), \(a\in \A_{det}\), and the stochastic vector fields \(f_A\), \(A\in \A_{stoch}\), in \eqref{eq:SDE} are assumed to be globally Lipschitz, thus guaranteeing existence and uniqueness of the initial value problem for \eqref{eq:SDE} itself  and for the split systems.
The theorems below provide bounds for the weak local error and the mean square local error.

We begin with weak approximations. The third hypothesis used below is the same as inequality (2.17) in \cite{milstein} which is key in establishing Theorem 2.5 in that reference. The first and second hypotheses just make explicit the differentiability  requirements on \(f_a\), \(f_A\), and \(\chi\) that have to be imposed to guarantee that \(D_w\chi\) makes sense when \(w\) has weight \(\nu+1\).

\begin{theo}Let \( \nu\) be a positive integer. Assume that:
\begin{itemize}
\item The deterministic vector fields \(f_a\), \(a\in \A_{det}\), are  of class
\( C^{2\nu}\), while the stochastic vector fields \(f_A\), \(A\in \A_{stoch}\), are of class
\( C^{2\nu+1}\).
\item The observable \(\chi\) is of class \(C^{2\nu+2}\) in \(\R^d\).
\item There is  a constant \( C>0 \) such that for each \( x\in\R^d \) and each word \( w \) of weight \(\nu+1\):
\[
| D_w\chi(x)| \leq C (1+|x|^2)^{1/2}.
\]
\item The weak error  conditions \eqref{eq:oc2} hold.
\end{itemize}
Then there exists  a constant \( K>0\) such that for each \(x_0\), each \(t_0\geq 0\) and each \(h>0\):
\[
|\E\chi\big( \widetilde\phi_{t_0+h,t_0}(x_0)\big)-\E\chi\big(\phi_{t_0+h,t_0}(x_0)\big)|
 \leq K (1+|x_0|^2)^{1/2} h^{\nu+1}
\]
(the expectation is conditional on \(x_0\)).
\end{theo}
\begin{proof}
 Define the residuals
\[
R_{t_0,h}(x_0) = \chi\big(\phi_{t_0+h,t_0}(x_0)\big)- \sum_{n\in \N/2,\atop {n\leq\nu}}\:\: \sum_{\| w\| = n} J_w(t_0,h) D_w \chi (x_0)
\]
and
\[
\widetilde R_{t_0,h}(x_0) = \chi\big(\widetilde\phi_{t_0+h,t_0}(x_0)\big)- \sum_{n\in \N/2,\atop {n\leq\nu}}\:\: \sum_{\| w\| = n} \widetilde J_w(t_0,h) D_w \chi (x_0)
\]
associated with the true and numerical solution respectively. If the weak order conditions hold, we have, after using the fifth item in Proposition \ref{prop:J} and is counterpart in Proposition \ref{prop:Jtilde},
\[
\E\chi\big( \widetilde\phi_{t_0+h,t_0}(x_0)\big)-\E\chi\big(\phi_{t_0+h,t_0}(x_0)\big) = \E \widetilde R_{t_0,h}(x_0)-
\E  R_{t_0,h}(x_0)
\]
and our task is to successively bound the two terms the right hand-side.

For the theoretical solution, the standard stochastic Taylor expansion (see e.g.\ \cite[Section 5.6]{kloeden} or \cite[Section 1.2]{milstein}) provides the following representation as an iterated Stratonovich integral
\begin{eqnarray*}
 R_{t_0,h}(x_0) &=& \sum_{w} \int_{t_0}^{t_0+h} \circ d B_{\ell_r}(s_r)\int_{t_0}^{s_r} \circ d B_{\ell_{r-1}}(s_{r-1}) \cdots \\&&\qquad\qquad\qquad\qquad\int_{t_0}^{s_2}
\circ d B_{\ell_1}(s_1) D_w\chi\big(\phi_{s_1,t_0}(x_0)\big);
\end{eqnarray*}
here the \(\ell_i\) are deterministic or stochastic letters, the sum is extended to all words of the form  \( w= \ell_1\dots\ell_r\), where \(\|\ell_2\dots\ell_r\|= \nu\) and it is understood that, for a deterministic letter \(\ell_i\), \(dB_\ell(s_i)\) means \(ds\). We next rewrite the iterated Stratonovich integrals as combinations of iterated Ito integrals as in \cite[Remark 5.2.8]{kloeden}; in each  resulting iterated integral the sum of the weights of the letters of the Brownian motions that appear is \(\nu+1\).
An application of \cite[Lemma 2.2]{milstein} then shows that, for a suitable constant \(L\),  \( \E | R_{t_0,h}(x_0)|^2 \leq L^2 (1+|x_0|^2)
h^{2\nu+2}\),  which implies \( \E | R_{t_0,h}(x_0)| \leq L (1+|x_0|^2)^{1/2}h^{\nu+1}\).

We now turn to the residual in the numerical solution. As in the proof of Theorem 4 in \cite{words}, we observe that, given an initial condition \(x_0=x(t_0)\) and any splitting algorithm,  the numerical solution after one step  \( t_0\rightarrow t_0+h\) is the same as the value of true solution at \(t_0+h\) of  a time-dependent SDE in which the originally given vector fields  are switched on and off as time evolves. For instance, in the simplest case where the SDE is \( dx = f_a(x)dt+f_A(x)\circ dB_A(t)\) and the (Lie-Trotter) numerical scheme consists of advancing with  \( dx = f_A(x)\circ dB_A(x)\) and then with \( dx = f_a(x)dt\), the time-dependent SDE is
\[ dx = 1_{ \{ t_0+h/2< t\leq t_0+h\} } f_a(x)2dt+1_{\{t_0\leq t\leq t_0+h/2\}}f_A(x)\circ dB_A(t_0+2(t-t_0)),
\]
where \(t_0\leq t\leq t_0+h\) and \( 1_{\{\cdot\}} \) denotes an indicator function. Using this observation the numerical residual may be bounded by reproducing the steps taken above to bound the residual of the true solution.
\end{proof}

The last result refers to the mean square error. The proof is parallel to that we have just presented and will be omitted.

\begin{theo}Let \( \mu\) be a positive integer multiple of \(1/2\). Assume that:
\begin{itemize}
\item The deterministic vector fields \(f_a\), \(a\in \A_{det}\), are  of class
\( C^{2\mu}\), while the stochastic vector fields \(f_A\), \(A\in \A_{stoch}\), are  of class
\( C^{2\mu+1}\).
\item There is  a constant \( C>0 \) such that for each \( x\in\R^d \) and each word \( w \) of weight \(\mu+1\):
\[
| f_w(x)| \leq C (1+|x|^2)^{1/2}.
\]
\item The strong error  conditions \eqref{eq:oc1} hold.
\end{itemize}
Then there exists  a constant \( K>0\) such that for each \(x_0\), each  \(t_0\geq 0\) and each \(h>0\):
\[
\Big(\E|\big( \widetilde\phi_{t_0+h,t_0}(x_0)-\phi_{t_0+h,t_0}(x_0)|^2
\Big)^{1/2} \leq K (1+|x_0|^2)^{1/2} h^{\mu+1/2}
\]
(the expectation is conditional on \(x_0\)).
\end{theo}

These local error bounds, in tandem with standard results (see e.g.\ \cite{milstein}), lead to bounds for the {\em global error}. For instance for the Langevin equations considered in Section 5,  the order conditions are fulfilled with \(\mu = 1\) and \(\nu= 2\). It then follows that both integrators are convergent with mean square global errors \(\mathcal{O}(h)\) and weak global errors \(\mathcal{O}(h^2)\) if the force \( F(q)\) satisfies the corresponding smoothness and growth hypotheses.

\bigskip

{\bf Acknowledgement.} We are thankful to
Chuchu Chen and Xu Wang for some useful discussions.
J.M. Sanz-Serna has been supported by project MTM2013-46553-C3-1-P  from Ministerio de Eco\-nom\'{\i}a y Comercio, Spain. A. Alamo has been supported by Universidad de Valladolid and IMUVA

\end{document}